\documentclass[11pt]{amsart}
\usepackage{latexsym,amsfonts,amsmath,amssymb} 
\usepackage{upref}
\usepackage{color}
\newtheorem{Satz}{Satz}[section]
\newtheorem{Thm}[Satz]{Theorem}

\newtheorem{Corol}[Satz]{Corollary}
\newtheorem{Prop}[Satz]{Proposition}
\theoremstyle{remark}

\theoremstyle{definition} 
\newtheorem{Def}[Satz]{Definition}
\newtheorem{Rem}[Satz]{Remark}
\numberwithin{equation}{section}
\newcommand{\Lrn}{{\mathcal{L}(\rn)}}
\newcommand{\hra}{\ensuremath{\hookrightarrow}}
\newcommand{\Lrp}{\ensuremath{{ L}^{r}_{p}(\rn)}}
\newcommand{\Lmwe}{\ensuremath{{ L}^{r}_{p}(w,\rn)}}
\newcommand{\SLmwe}{\ensuremath{{ L}^{r}_{p}(\lambda,w,\rn)}}
\newcommand{\Lpw}{\ensuremath{L_{p,w} (\rn)}}
\newcommand{\Lpzw}{\ensuremath{L_{p_0,w} (\rn)}}
\newcommand{\rn}{\ensuremath{{\mathbb R}^n}}
\newcommand{\real}{\mathbb{R}}

\newcommand{\eq}{equation}
\newcommand{\supp}{\ensuremath{\mathrm{supp \,}}}
 \newcommand{\Dcal}{\mathcal{D}}
\newcommand{\norm}[2]{\left\Vert\left. #1 \right| #2 \right\Vert}

\begin{document}
\title[Boundedness on Morrey spaces, extrapolation and embeddings]
{Extension and boundedness of operators\\ on Morrey spaces\\ 
from extrapolation techniques and embeddings}

\author{Javier Duoandikoetxea and Marcel Rosenthal} 
\address{Universidad del Pa\'is Vasco/Euskal Herriko Unibertsitatea, Departamento de Matem\'a\-ti\-cas/Matematika saila, 
Apdo. 644, 48080 Bilbao, Spain}

\email{javier.duoandikoetxea@ehu.eus, marcel.rosenthal@uni-jena.de} 
\thanks{The first author is supported by the grants MTM2014-53850-P of the Ministerio de Econom\'{\i}a y Competitividad (Spain) 
and grant IT-641-13 of the Basque Gouvernment. The second author is supported by the German Academic Exchange Service (DAAD)}


\begin{abstract}
We prove that operators satisfying the hypotheses of the extrapolation theorem for Muckenhoupt weights are bounded on weighted Morrey spaces. As a consequence, we obtain at once a number of results that have been proved individually for many operators. On the other hand, our theorems provide a variety of new results even for the unweighted case because we do not use any representation formula or pointwise bound of the operator as was assumed by previous authors. To extend the operators to Morrey spaces we show different (continuous) embeddings of (weighted) Morrey spaces into appropriate Muckenhoupt $A_1$ weighted $L_p$ spaces, which enable us to define the operators on the considered Morrey spaces by restriction. In this way, we can avoid the delicate problem of the definition of the operator, often ignored by the authors. In dealing with the extension problem through the embeddings (instead of using duality) one is neither restricted in the parameter range of the $p$'s (in particular $p=1$ is admissible and applies to weak-type inequalities) nor the operator has to be linear. Another remarkable consequence of our results is that vector-valued inequalities in Morrey spaces are automatically deduced. On the other hand, we also obtain $A_\infty$-weighted inequalities with Morrey quasinorms.
\end{abstract}

\maketitle

\section{Introduction}\label{sec1}
An important tool of modern harmonic analysis is the extrapolation theorem due to Rubio de Francia. 
For a given operator $T$ we suppose that for some $p_0$, $1 \le p_0 < \infty$, and
for every weight belonging to the Muckenhoupt class $A_{p_0}$ the inequality
\begin{\eq} \label{i1}
  \norm{Tf}{\Lpzw}= \left(\int_{\rn} |Tf(x)|^{p_0} w(x) dx \right)^\frac{1}{p_0} \le c \norm{f}{\Lpzw}
\end{\eq}
holds, where the constant $c$ is independent of $f$ and depends on $w$ only in terms of its $A_{p_0}$-constant,  $[w]_{A_{p_0}}$ (see Definition \ref{defap}). 
Then for every $p$, $1 < p < \infty$, and every $w \in A_{p}$, there exists a constant depending
on $[w]_{A_{p}}$ such that
\begin{\eq} \label{i2}
  \norm{Tf}{\Lpw}\le c \norm{f}{\Lpw}
\end{\eq}
(cf.\ \cite{Ru1, Ru2} for the original results, or \cite{CMP11, D11} for modern presentations).
One fact which makes extrapolation theory so powerful is that one does not need any assumption on $T$ besides having $T$ well-defined on $\bigcup_{w\in A_{p_0}} L_{{p_0},w}(\rn)$. (Actually one can extrapolate just inequalities between pairs of functions, not even an operator is needed.)
Even if one is not interested in any weighted results this is very interesting by the fact that for example the weighted $L_2$-boudednesses implies in particular unweighted boundednesses in the complete $L_p$ scale.
Appropriate assertions hold also for subclasses of $A_{p_0}$, that is, one assumes \eqref{i1} for a subclass (e.g. $A_1$) and deduces \eqref{i2} for appropriate subclasses (this applies to different Fourier multipliers, for instance).

We generalize results of this theory to Morrey spaces and will get even in the classical Morrey spaces new results on the boundednesses of operators in the sense that we do not need any condition on $T$ except that $T$ satisfies \eqref{i1} (for weights of $A_{p_0}$ or distinguished subclasses) and is well-defined on $\bigcup_{w\in A_{p_0}} L_{{p_0},w}(\rn)$. 

There are lots of papers dealing with classes of operators in various Morrey type spaces which admit size estimates as 
\begin{\eq} \label{i5}
  |(Tf)(y)|\leq c \int_{\real^n} \frac{|f(x)|}{|y-x|^n} dx 
\end{\eq}
for all $f\in \Dcal(\rn)$ and $y\notin \supp(f)$, where $\Dcal(\rn)$ are the compactly supported smooth functions (cf. \cite{Pee66,ChFr87,N94, Alv96, Y98, Sam09,KS09, G11, Gul12, Mu12, SFZ13, RS14, KGS14, PT15,IST15, Nak16,Wan16} and the references given there).
To obtain the boundedness results they are explicitly using and requiring different representation formulas as \eqref{i5} of the considered operators.
We are completely avoiding these assumptions and encapsulating them through the consideration of weighted spaces. Hence, we get even new results on the classical Morrey spaces regarding Mihlin-H\"ormander type operators, rough operators, pseudodifferential operators, square functions, commutators, Fourier integral operators\ldots
\par
Moreover, literature about (the nonseparable) Morrey spaces (starting from Peetre \cite{Pee66} and many following scholars) does not care about how to extend the considered operators (satisfying a required representation formula only on nondense subsets of the considered Morrey spaces).
Some forerunners dealing with the extension of the operators are \cite{Alv96, AX12, RT13, Ros13,  RoT14_2, T-HS, RS14, Ad15, RS16}.
Mainly, they apply duality (with some exceptions as \cite{T-HS, RS14} justifying \eqref{i5} also for Morrey functions), which restricts the boundedness results to linear operators and moreover the range of the integration parameter $p$ (in particular, the extension problem was not covered for $p=1$ and the corresponding weak-type inequalities).
We show different (continuous) embeddings of (weighted) Morrey spaces into different Muckenhoupt $A_1$ weighted $L_p$ spaces, which enables us to define the operators on the considered Morrey spaces by restriction. 

Another possibility to overcome the extension problem is using the Fatou property of the Morrey space. Nevertheless, this method requires at least a bigger space where the operator is also bounded to extend it by restriction (cf.~Triebel \cite{TrFat}), which is exactly the type of our embeddings. 
However, by the fact that we do not require representation formulas of operators it is even sufficient to show a not necessarily continuous embedding of a Morrey type space into a set of the type $\bigcup_{w\in A_{p}} L_{{p},w}(\rn)$ (where the considered operator is required to be bounded). This can be seen as a powerful model case  to extend operators which are a priori only defined on $\Dcal(\rn)$ ~---as multipliers and (maximally truncated) singular integrals--- to the various generalizations of Morrey spaces given in the literature and to overcome the extension problem (by the fact that we avoid the requirement of justifying \eqref{i5} also for Morrey functions). 
If one wants to deal with other spaces where $\Dcal(\rn)$ is not dense and which are contained in the set $\bigcup_{w\in A_{p}} L_{{p},w}(\rn)$ and where the maximal operator is bounded (or in their appropriate associated spaces), then there are even very good chances that the presented results could be carried over partially.

A extrapolation result for Morrey spaces appears also in \cite{RS16}. Its proof uses the boundedness of the Hardy-Littlewood maximal operator on predual Morrey spaces, which are also used to define the operator by duality. We avoid this in our approach and obtain more general results, valid also for end-points and for limited-range type assumptions.   

We also extrapolate from inequalities with $L_{p,w}$ (quasi)-norms for $0<p<\infty$ and $w\in A_\infty$ to Morrey (quasi)-norms. In this way we get new Morrey versions of several interesting inequalities of this type appearing in the literature for pairs of operators.

We consider only Muckenhoupt weights because the extrapolation techniques are adapted to them and because there are many operators for which the boundedness on Lebesgue spaces with (subclasses of) Muckenhoupt weights are known, so that we can immediately infer Morrey estimates from our general results. Nevertheless, a description of the admissible weighted Morrey estimates is not known even for the Hardy-Littlewood maximal operator and, in fact, in the case of the Morrey spaces $\SLmwe$ (see Definition \ref{d1:def}) the class of admissible weights goes beyond the Muckenhoupt $A_p$ class as shown in \cite{Ta15}. A similar situation holds for the Hilbert transform (\cite{Sam09}) and more generally for the Riesz transforms and other singular integrals (\cite{NkSa17}). On the other hand, in these cases the problem of defining the operators in the corresponding weighted spaces, which in our case is solved by the embeddings, should be considered. 

After introducing the notation and some preliminaries in Section \ref{sec2}, Section \ref{sec3} is concerned with continuous embeddings of Morrey type spaces into Muckenhoupt weighted $A_1$ Lebesgue spaces. 
In Section \ref{sec4} we present the extrapolation theory generalized to Morrey type spaces which will be applied to far-reaching boundedness results given in Section \ref{sec5}. 
\section{Notation, Morrey spaces and preliminaries}\label{sec2}

A weight is a locally integrable nonnegative function. For $0< p <\infty$, $L_{p,w} (\rn)$ is the complex quasi-Banach space of functions whose $p$-th power is integrable with respect to the weight $w$, with the quasinorm given by
\[
\| f \, | L_{p,w} (\rn) \| = \Big( \int_{\rn} |f(x)|^p \,  w(x) dx \Big)^{1/p}= \Big( \int_{\rn} |f|^p \,  w \Big)^{1/p}.
\]
Moreover, we write $w(M)=\int_{M} w(x) dx$ for the measure of the subset $M$ of $\rn$. We similarly define $L_{p,w} (M)$. 
If $w(x)\equiv 1$, we simply write $L_{p}(M)$, $\| \cdot \, | L_{p} (M) \|$ and $|M|$. 
Furthermore, 
$\chi_M$ denotes the characteristic function of $M$. 
For any $p\in (1,\infty)$ we denote by $p'$ the conjugate index, namely, $1/p+1/{p'}= 1$. If $p=1$, then $p'=\infty$.
Moreover, $\mathcal{L}(\rn)$ collects all equivalence classes of almost everywhere coinciding measurable complex functions. A subset of  $\mathcal{L}(\rn)$ is $L_1^{\mathrm{loc}}(\rn)$, which collects all locally integrable functions, that is, functions in $L_1 (M)$ for any bounded measurable set $M$ of $\rn$. 
We consider cubes whose sides are parallel to the coordinate axes. For such a cube $Q$,  $\delta Q$ stands for the concentric cube with side-length $\delta$ times of the side-length of $Q$.
The concrete value of the constants may vary from one
formula to the next, but remains the same within one chain of (in)equalities. 

\begin{Def} \label{defap}
Let 
$w\in L_1^\textrm{loc}(\rn)$ with $w>0$ almost everywhere. We say that $w$ is a \textit{Muckenhoupt weight} belonging to $A_p$ for $1<p<\infty$ if
\begin{\eq*} 
  [w]_{A_p}\equiv 
	\sup_Q \frac{w(Q)}{|Q|} \left( \frac{w^{1-p'}(Q)}{|Q|} \right)^{p-1}<\infty,
\end{\eq*}
where the supremum is taken over all cubes $Q$ in $\real^n$. The quantity  $[w]_{A_p}$ is the $A_p$ constant of $w$.
We say that $w$  belongs to $A_1$ if, for any cube $Q$,
\begin{equation*} 
  \frac{w(Q)}{|Q|}\le c w(x) \text{ for almost all } x\in Q.
\end{equation*}
The $A_1$ constant of $w$, denoted by  $[w]_{A_1}$, is the smallest constant $c$ for which the inequality holds. 

We say that $w$ is in $A_\infty$ if $w\in A_p$ for some $p$.
\end{Def}
\begin{Rem}
The $A_p$ weights are doubling. This means that if $w\in A_p$, there exists a constant $c$ such that 
\begin{equation} \label{DC} 
 w(2Q)\le cw(Q),
\end{equation} 
 for every cube $Q$ in $\rn$ (cf.~\cite[(7.3)]{D01}). 
\end{Rem}
We recall that the Hardy-Littlewood maximal operator, which we denote by $M$, is bounded on $L_{p,w} (\rn)$ for $1< p <\infty$ if and only if $w\in A_p$ and is of weak-type $(1,1)$ with respect to the measure $w(x)dx$ if and only if $w\in A_1$. On the other hand, we will use the following construction of $A_1$ weights: if $Mf(x)$ is finite a.e.~and $0\le\delta<1$, then $Mf(x)^\delta$ is an $A_1$ weight whose $A_1$-constant depends only on $\delta$. Moreover, the factorization theorem says that $w\in A_p$ if and only if there exist $w_0, w_1\in A_1$ such that $w=w_0w_1^{1-p}$ (cf. \cite{Ru2} or \cite[Thm.~9.5.1]{G09}).  

We define some Muckenhoupt weighted Morrey spaces.
\begin{Def}{} \label{d1:def}
\upshape 
	For $0< p< \infty$, $-\frac{n}{p}\leq r<0$ and a weight $w$ we define 
	\begin{equation*}
				\Lmwe\equiv\{f \in \Lrn \mbox{ : } \left\|f|\Lmwe\right\|<\infty\}
	\end{equation*}
	with the quasinorm
	\begin{align*} 
		\left\|f|\Lmwe\right\|\equiv & 
\sup_{Q} w(Q)^{-\left(\frac{1}{p}+\frac{r}{n} \right)} \left\|f|L_{p, w}(Q)\right\|,
	\end{align*}	
where the supremum is taken over all cubes $Q$ in $\real^n$.	
Moreover,
		\begin{equation*}
				\SLmwe\equiv\{f \in \Lrn \mbox{ : } \left\|f|\SLmwe\right\|<\infty\}
	\end{equation*}
	with the quasinorm  
	\begin{align*} 
		\left\|f|\SLmwe\right\|\equiv & 
\sup_{Q} |Q|^{-\left(\frac{1}{p}+\frac{r}{n} \right)} \left\|f|L_{p, w}(Q)\right\|.
	\end{align*}	
We also define their weak versions given by the quasinorms
	\begin{align*} 
		&\left\|f|W\Lmwe\right\|\equiv  
\sup_{Q} \sup_{t>0} w(Q)^{-\left(\frac{1}{p}+\frac{r}{n} \right)}  t w(\{x\in Q:\ |f(x)|>t\})^{\frac{1}{p}} ,
		\\ &\left\|f|W\SLmwe\right\|\equiv 
\sup_{Q} \sup_{t>0} |Q|^{-\left(\frac{1}{p}+\frac{r}{n} \right)} t w(\{x\in Q:\ |f(x)|>t\})^{\frac{1}{p}}.
	\end{align*}	
\end{Def}
\begin{Rem}
We observe that ${L}^{-{n}/{p}}_p(w,\rn)={L}^{-{n}/{p}}_p(\lambda,w,\rn)=L_{p, w}(\rn)$. If $w\equiv 1$, $\Lmwe$ and $\SLmwe$ coincide with the unweighted Morrey space $\Lrp$. The weak Morrey spaces coincide also with the weak $\Lpw$ spaces for $r=-n/p$.
\end{Rem}

\begin{Def}\label{defRH}
\upshape 
Let a nonnegative locally integrable function $w$ on $\real^n$ belong to the \textit{reverse H\"older class} $RH_\sigma$ for $1<\sigma<\infty$ if it satisfies the \textit{reverse H\"older inequality} with exponent $\sigma$, i.e. \begin{\eq*} 
  \left(\frac{1}{|Q|}\int_{Q} w(x)^\sigma dx \right)^\frac{1}{\sigma}\le \frac{c}{|Q|} \int_{Q} w(x) dx,
\end{\eq*}
where the constant $c$ is universal for all cubes $Q\subset \rn$.
\end{Def}

The $RH_\sigma$ classes are decreasing, that is, $RH_\sigma \subset RH_\tau$ for $1<\tau<\sigma<\infty$. On the other hand, Gehring's lemma (\cite{gehring}) says that if $w\in RH_\sigma$, there exists $\epsilon>0$ such that $w\in RH_{\sigma+\epsilon}$ (openness of the reverse H\"older classes).

\begin{Rem} Let $w\in RH_\sigma$. For any cube $Q$ and any measurable $E\subset Q$ it holds that 
\begin{equation}\label{ineqRH}
 \frac{w(E)}{w(Q)}\le c \left(\frac{|E|}{|Q|}\right)^{1/\sigma'}. 
\end{equation}
(See \cite[Cor.~7.6]{D01}). Since $w\in A_p$ implies that $w\in RH_\sigma$ for some $\sigma$ (cf. \cite[Thm.~7.4]{D01}), the inequality holds for each $A_p$ weight for the appropriate $\sigma$. Consequences of the reverse H\"older inequalities for weights are that if $w\in A_p$, then there exist $\epsilon>0$ and $s>1$ such that $w\in A_{p-\epsilon}$ (openness of $A_p$ classes) and $w^s\in A_p$.
\end{Rem}

 In the paper we use sometimes classes of the form $A_p\cap RH_\sigma$. There is a characterization of R.~Johnson and C.~Neugebauer for them (\cite{JN91}):
\begin{equation}\label{apandrh}
A_p\cap RH_\sigma =\{w: w^\sigma\in A_{\sigma(p-1)+1}\}.
\end{equation}
This is also valid for $p=1$, that is, $A_1\cap RH_\sigma =\{w: w^\sigma\in A_{1}\}$, which can be easily proved from the definition of $A_1$. 

\begin{Rem} \label{Ma3}
Let $w(x)=|x|^\alpha$ or $w(x)=(1+|x|)^\alpha$. Then we have the following:
\begin{enumerate}
\item It holds that $w \in A_p$ if and only if $-n<\alpha<n(p-1)$ whenever $p\in (1,\infty)$, and $-n<\alpha\le 0$ whenever $p=1$.
	\item For $\alpha \in (-n,0)$ it holds $w \in RH_\sigma$ if and only if $1<\sigma<-n/\alpha$. \label{Ma1}
	\item For $\alpha \ge 0$ it holds $w \in RH_\sigma$ for all $1<\sigma<\infty$. \label{Ma2}
\end{enumerate}
The first one is well known. The second one is easily obtained using \eqref{apandrh} (for $p=1$). The third one can be checked directly or derived from \eqref{apandrh}.
\end{Rem}

\section{Embeddings of Muckenhoupt weighted Morrey spaces into Muckenhoupt weighted Lebesgue spaces}\label{sec3}
In this section we establish different continuous embeddings between Morrey spaces of the types $\Lmwe$ and $\SLmwe$ into some Muckenhoupt $A_1$ weighted Lebesgue spaces which will give us the possibility to define operators on these Morrey spaces by restriction. Extension by continuity is usually not working in Morrey type spaces by their nonseparability and also by the fact that the smooth functions are not dense in these spaces (see \cite[Prop.~3.7]{RoT14_2} for $\Lrp$ and \cite[Prop.~3.2]{RS16} for $\Lmwe$ for the nonseparability, and see \cite[p.~22]{Pic69} for the nondensity). In recent literature this difficulty has been treated using duality for linear operators, whenever the integration parameter $p$ satisfies $1<p<\infty$. Dealing with the extension problem through embeddings one is neither restricted in the parameter range of $p$ (in particular, $p=1$ is admissible) nor the operator has to be linear.

\begin{Prop} \label{LEmb}
Let $1\le\gamma < p<\infty$, $-n/p \le r<0$ and $w\in A_{p/\gamma}$. 
Then there exists $q_0\in(\gamma, p)$ such that for each $q\in [\gamma,q_0]$,  $1<s\le s_0(q)$, each cube $Q$, and $h\in L_{(p/q)',w}(Q)$ with norm $1$, the continuous embedding
\begin{\eq} \label{emb1new}
  \Lmwe \hra L_{q,M(h^sw^s\chi_Q)^{1/s}}(\rn)
\end{\eq} 
holds with constant independent of $h$ and depending on $Q$ as $w(Q)^{1/p+r/n}$.

Furthermore, with the additional assumptions $w\in RH_{\sigma}$ and $r\le-n/(p\sigma)$, it also holds that
\begin{\eq} \label{emb2new}
  \SLmwe \hra L_{q,M(h^sw^s\chi_Q)^{1/s}}(\rn)
\end{\eq} 
 with constant independent of $h$ and depending on $Q$ as $|Q|^{1/p+r/n}$.

As a consequence, under the same conditions on $p$, $r$ and $w$, for $q\in [\gamma,q_0]$ there is $0<\alpha<n$ such that
\begin{\eq} \label{emb1}
  \Lmwe, \SLmwe \hra L_{q,(1+|x|)^{-\alpha}}(\rn).
\end{\eq}
\end{Prop}

\begin{proof} We prove the result for $\gamma=1$. The case $\gamma>1$ will follow as a consequence.

Choose $q_0\in (1,p)$ such that $w\in A_{p/q_0}$. For $q\in [1,q_0]$, $w\in A_{p/q}$. Fix one such $q$ and set $\tilde p=p/q$. Let $1<s<\tilde p'$ such that $w^{1-\tilde p'}\in A_{\tilde p'/s}$. 

For a function $h$ with $\Vert h| L_{\tilde p',w}(Q)\Vert=1$, the weight $M(h^s w^s \chi_{Q})^{1/s}$ is well defined because $h^s w^s \chi_{Q} \in L_1(\rn)$, which implies $M(h^s w^s \chi_{Q}) < \infty$ almost everywhere (by the fact that $M$ maps $L_1(\rn)$ to the weak $L_1(\rn)$ space). We check that  $h^s w^s \chi_{Q} \in L_1(\rn)$ and obtain an estimate that will be needed below. Indeed,
\begin{align} 
\begin{split} \label{p3}
	 \left(\int_{Q} h^s w^{s-1} w\right)^\frac{1}{s}
	&\le\norm{h}{L_{\tilde p',w}(\rn)}
	\left(\int_{Q} w^{(s-1) \frac{\tilde p'}{\tilde p'-s} +1} \right)^{\frac1{s}-\frac{1}{\tilde p'}}\\ 
	& \le c\  |Q|^{\frac{1}{s}}w^{1- \tilde p'}(Q)^{-\frac{1}{\tilde p'}}
		\le c\ w(Q)^{\frac{1}{\tilde p}} |Q|^{-\frac{1}{s'}},
\end{split} 
\end{align}
where the second inequality holds because $w^{1-\tilde p'}\in A_{\tilde p'/s}$ (the exponent of $w$ in the integral is the same as $(1-\tilde p')(1-(\tilde p'/s)')$) and in the last one we use 
\begin{equation*}
|Q|\le w(Q)^{\frac{1}{\tilde p}} w^{1- \tilde p'}(Q)^{\frac{1}{\tilde p'}}.
\end{equation*}

Let $f\in  \Lmwe$ and assume that $f$ is nonnegative. To estimate the norm of $f$ in $L_{q,M(h^sw^s\chi_Q)^{1/s}}(\rn)$ we decompose the integral into two parts: first we integrate over $2Q$ and next over $\rn\setminus 2Q$. Over $2Q$ we use H\"older's inequality,
\begin{equation*} 
  \left(\int_{2Q} f^{q}  M(h^s w^s \chi_{Q})^\frac{1}{s}\right)^\frac 1q  \le 
	\left(\int_{2Q} f^{p}w\right)^{\frac1{p}}\left(\int_{2Q} M(h^s w^s \chi_{Q})^{\frac{\tilde p'}s}w^{1- \tilde p'}\right)^{\frac 1{q\tilde p'}}.
\end{equation*}
The first term of the right-hand side is bounded by $w(2Q)^{\frac{1}{p}+\frac{r}{n}} \norm{f}{\Lmwe}$ and using \eqref{DC} we can replace $w(2Q)$ by $w(Q)$. In the second one we use that $w^{1-\tilde p'}\in A_{\tilde p'/s}$ and the boundedness of $M$ on $L_{\tilde p'/s,w^{1-\tilde p'}}(\rn)$ to get a constant times the norm of $h$ in $L_{\tilde p',w}(\rn)$, which is $1$.

To handle the integral over $\rn\setminus 2Q$, we use that $M(h^s w^s \chi_{Q})(x)$ for $x\in 2^{i+1} Q\setminus 2^i Q$ ($i\ge 1$) is comparable to the average of $h^s w^s \chi_{Q}$ over $ 2^{i+1} Q$. 
Then we obtain 
\begin{align} 
\begin{split} \label{p4}
	& \int_{\rn\setminus 2Q} f^{q}M(h^s w^s \chi_{Q})^\frac{1}{s}
  \le c_1 \sum_{i=1}^\infty \int_{2^{i+1} Q\setminus 2^i Q} f^{q}\left( \frac{\int_{Q} h^s w^s}{|2^{i}Q|} \right)^\frac{1}{s}
  \\ &\le  c_2 \sum_{i=1}^\infty \int_{2^{i+1} Q\setminus 2^i Q} f^{q} w^{\frac 1{\tilde p}} w^{-\frac 1{\tilde p}}\frac{w(Q)^{\frac 1{\tilde p}}}{|2^{i}Q|^\frac{1}{s} |Q|^{\frac{1}{s'}}} 
  \\ &\le  c_2 \sum_{i=1}^\infty \left(\int_{2^{i+1} Q\setminus 2^i Q} f^{p} w \right)^\frac{q}{p} w^{1-\tilde p'} \left(2^{i+1}Q\right)^{\frac 1{\tilde p'}} \frac{w(Q)^{\frac 1{\tilde p}}}{2^{in/s}|Q|}
  \\ &\le  c_3 \norm{f}{\Lmwe}^q
	\sum_{i=1}^\infty w\left(2^{i+1} Q \right)^{\left(\frac{1}{p}+\frac{r}{n}\right)q} w^{1-\tilde p'} \left(2^{i+1}Q\right)^{\frac 1{\tilde p'}} \frac{w(Q)^{\frac 1{\tilde p}}}{2^{in/s}|Q|}.
\end{split} 
\end{align}
Since $w\in A_{\tilde p}$, then
\begin{equation}\label{p6b}
w\left(2^{i+1} Q \right)^{\frac 1{\tilde p}} w^{1-\tilde p'} \left(2^{i+1}Q\right)^{\frac 1{\tilde p'}}\le c |2^{i+1} Q|,
\end{equation}
and it is sufficient to show 
\begin{align} \label{p5}
 \sum_{i=1}^\infty w\left(2^{i+1} Q \right)^{\frac{r}{n}q} w(Q)^\frac{q}{p} 2^{\frac{in}{s'}} 
	\le c \, w\left( Q \right)^{\left(\frac{1}{p}+\frac{r}{n}\right)q}.
\end{align}
Taking into account \eqref{ineqRH}, there exists $\delta>0$ such that 
\begin{align*}
\begin{split}
	w(2^{i+1}Q)^{\frac{r}{n}q} \le c\, 2^{i\delta r q} w(Q)^{\frac{r}{n}q}.
\end{split} 
\end{align*}
Therefore, \eqref{emb1new} is satisfied for $s$ sufficiently close to $1$.

In the case of the spaces $\SLmwe$, the part corresponding to the integral over $2Q$ is obtained in a similar way using 
\[
 \left(\int_{2Q} f^{p}w\right)^{\frac1{p}}
	\le c |Q|^{\frac{1}{p}+\frac{r}{n}} \norm{f}{\SLmwe}. 
\]
For the integral over $\rn\setminus 2Q$, we first modify the last line of \eqref{p4} and then use \eqref{p6b} to get
\begin{align*} 
\begin{split} 
	& \int_{\rn\setminus 2Q} f^{q}M(h^s w^s \chi_{Q})^\frac{1}{s}
  \\ & \le  c_1 \norm{f}{\SLmwe}^q
	\sum_{i=1}^\infty \left|2^{i+1} Q \right|^{\left(\frac{1}{p}+\frac{r}{n}\right)q} w^{1-\tilde p'} \left(2^{i+1}Q\right)^{\frac 1{\tilde p'}} \frac{w(Q)^{\frac 1{\tilde p}}}{2^{in/s}|Q|}
	  \\ & \le  c_2 \norm{f}{\SLmwe}^q
	\sum_{i=1}^\infty \left|2^{i+1} Q \right|^{\left(\frac{1}{p}+\frac{r}{n}\right)q} 
	2^{\frac{in}{s'}} \left[w(2^{i+1}Q)^{-1}\ w(Q)\right]^\frac{q}{p}.
\end{split} 
\end{align*}
By the latter we obtain \eqref{emb2new} showing that
\begin{align} 
\begin{split} \label{p10}
	\sum_{i=1}^\infty 2^{inq[\frac{1}{p}+\frac{r}{n}+\frac{1}{s'q}]} 
	\left[w(2^{i+1}Q)^{-1}\ w(Q)\right]^\frac{q}{p} \le c.
\end{split} 
\end{align}
By the fact that $w\in RH_{\sigma}$, from \eqref{ineqRH} we deduce 
\begin{align} 
\begin{split} \label{p12}
  \frac{w(Q)}{w(2^{i}Q)}
	\le c \left(\frac{|Q|}{|2^{i}Q|}\right)^{1/\sigma'}
	= c  2^{-i\frac n{\sigma'}}.
\end{split} 
\end{align}
Inserting \eqref{p12} into the left hand-side of \eqref{p10} and taking into account that $s$ is as close to $1$ as desired, the geometric series is convergent for $r/n+1/(p\sigma)<0$. The endpoint $r=-n/(p\sigma)$ of the statement is attained from the openness property of the reverse H\"older classes.

Let now $1<\gamma<p$ and $f\in \Lmwe$ with $w\in A_{p/\gamma}$. Since
\begin{equation} \label{TFE}
	\norm{f}{\Lmwe}
	= \norm{|f|^\gamma}{L^{r\gamma}_{p/\gamma}(w,\rn)}^\frac{1}{\gamma},
\end{equation}
$|f|^\gamma\in L^{r\gamma}_{p/\gamma}(w,\rn)$ (with $w\in A_{p/\gamma}$) and we can apply the first part of the proof. It gives the existence of $q_0^*\in (1,p/\gamma)$ and $s>1$ such that for $q^*\in (1,q_0^*]$, 
\begin{equation*}
\norm{|f|^\gamma}{L^{r\gamma}_{p/\gamma}(w,\rn)}\le c \norm{|f|^\gamma}{L_{q^*,M(h^sw^s\chi_Q)^{1/s}}(\rn)},
\end{equation*}
for any cube $Q$ and $h\in L_{(p/(\gamma q*))',w}(Q)$ with norm $1$. Define $q_0=q_0^*\gamma$ and $q=q^*\gamma$ and \eqref{emb1new} is proved. The case $f\in\SLmwe$ with $w\in A_{p/\gamma}$ is treated similarly.

To obtain \eqref{emb1}, let $Q$ be the unit cube and $h=cw^{-1}$, with $c$ chosen such that $h$ has unit norm. Use now that $M(\chi_Q)(x)$ is bounded below by a positive constant if $x\in 2Q$ and by $c|x|^{-n}$ if $x\notin 2Q$. Take $\alpha=n/s$ to conclude. 
\end{proof}
\begin{Rem} \label{3.2T} Since $w\in A_p$ implies that $w\in RH_\sigma$ for some $\sigma>1$, the range of values of $r$ satisfying the embedding \eqref{emb1} for $\SLmwe$ is not empty.
In the case of the weights $|x|^\beta$ and $(1+|x|)^\beta$ such range is
\[
  \begin{cases}  
	  -\frac{n}{p} \le r<0, & \text{ if }\ 0\le \beta<n(\frac p{\gamma}-1); \\
		-\frac{n}{p} \le r<\frac{\beta}{p}, &\text{ if }\ -n<\beta<0. 	
	\end{cases}
\]
\end{Rem}

The following proposition shows that the union of the $L_{q,w}(\rn)$ spaces for fixed $q$ and all $w\in A_q$ is invariant for $1<q<\infty$. This result appears in \cite{KMM16} in a more general context. We give here a direct proof.

\begin{Prop} \label{DfOp}
Let be $1\le \gamma< q<\infty$. 
Then it holds that
\begin{\eq} \label{DefOp}
  \bigcup_{w\in A_{q/\gamma}} L_{q,w}(\rn) = \bigcup_{\gamma<p<\infty} \bigcup_{w\in A_{p/\gamma}} \Lpw 
	 \subsetneq \bigcup_{w\in A_1} L_{\gamma,w}(\rn) . 
\end{\eq}
Hence, the left-hand side is independent of $q$.
\end{Prop}
\begin{proof} Again it is enough to prove the case $\gamma=1$. Once this is proved, the general case is deduced from the observation that $f\in \Lpw$ with $w\in A_{p/\gamma}$ is equivalent to $|f|^\gamma\in L_{p/\gamma, w}(\rn)$ (with $w\in A_{p/\gamma}$).

Let  $\gamma=1$. We start showing the equality in \eqref{DefOp}. 
	We first  prove that if $f$ is in $\Lpw$ for $1<p<q$ and $w\in A_p$, then $f\in \bigcup_{v\in A_q} L_{q,v}(\rn)$.
	For $s>1$ sufficient small $w\in A_{p/s}$ holds. Moreover, we have $f^s \in L_{p/s,w}$ and therefore $M(f^s)\in L_{p/s,w}$. Hence $(M(f^s))^{1/s}\in A_1$. Thus, $(M(f^s))^{(p-q)/s}w\in A_{q}$ (see \cite[Lemma 2.1]{D11}).  
	By reason of $p<q$ finally it holds that
	\[
	  \int_{\rn} |f|^q (M(f^s))^{\frac{1}{s} (p-q)}w \le \int_{\rn} |f|^q |f|^{p-q}w
	\]
and $f\in \bigcup_{v\in A_q} L_{q,v}(\rn)$. 
That means that we have up to now shown
\begin{\eq}\label{half}
  \bigcup_{w\in A_q} L_{q,w}(\rn) = \bigcup_{1< p\le q} \bigcup_{w\in A_p} \Lpw.
\end{\eq}	
We assume now $f\in \Lpw$ for some $p>q$ and $w\in A_p$. By \eqref{emb1} (applied to $r=-n/p$, since ) we obtain $f\in L_{\tilde q,(1+|x|)^{-\alpha}}(\rn)$ for some $\tilde q<q$ and so $f$ is in the right-hand side of \eqref{half}. Here, $\tilde q=1$ is also admitted.

It remains to prove that the inclusion of \eqref{DefOp} is strict. For this it is enough to notice that there exists  $f\in L^1(\rn)$ for which $Mf$ is not locally integrable. Indeed, such an $f$  cannot be in $L_{p,w}(\rn)$ for any $p\in (1,\infty)$ because this would imply $Mf \in L_{p,w}(\rn)\subset L_1^{\textrm{loc}}(\rn)$. A typical example of a function as the one required here is 
$f(x)= |x|^{-n} (\ln{|x|})^{-2}\ \chi_{\{|x|\le 1/2\}}$. 
\end{proof}

The following embedding result on the considered Morrey spaces is a generalization of \cite[Prop.~2.10]{T-HS} to weighted situations.
\begin{Prop} \label{LEmbE}
Let $1\le p<\infty$ and $\nu \ge 1$. Let $w\in A_1\cap RH_\sigma$ (that is, $w^\sigma\in A_{1}$) for some $\sigma>\nu$. Then for $r\in [-n/p, -n\sigma'/(p\nu')]$ (if $\nu=1$, the endpoint $r=0$ is excluded), there exists $s_0>\nu$ such that for $\nu<s\le s_0$ and for any cube $Q$, the continuous embedding 
\begin{\eq} \label{emb1Enew}
  \Lmwe \hra L_{p,M(w^s\chi_Q)^{1/s}}(\rn) 
\end{\eq}
holds with a constant depending on $Q$ as $w(Q)^{1/p+r/n}$.

Moreover, whenever  
$r\in [-n/p,-n(1/\sigma + 1/\nu')/p]$, it also holds that
\begin{\eq} \label{emb2Enew}
  \SLmwe \hra L_{p,M(w^s\chi_Q)^{1/s}}(\rn), 
\end{\eq}
with a constant depending on $Q$ as $|Q|^{1/p+r/n}$.

As a consequence, under the same conditions on $w$ and $r$, there exists some $0<\alpha<n/\nu$ such that
\begin{\eq} \label{emb1E}
  \Lmwe,  \SLmwe \hra L_{p,(1+|x|)^{-\alpha}}(\rn). 
\end{\eq}
\end{Prop}

\begin{proof}
Let $s$ be such that $\nu<s<\sigma$. The weight $M(w^s\chi_Q)^{1/s}$ is in $A_1$ and satisfies  $M(w^s\chi_Q)^{1/s}\le M(w^s)^{1/s}\le C w$ a.e.\ because $w^s\in A_1$. Moreover, for $x\in 2^{i+1} Q\setminus 2^i Q$, $M(w^s \chi_{Q}) (x)$ is comparable to the average of $w^s \chi_{Q}$ over $2^{i+1} Q$. 

Let $f\in  \Lmwe$ with $r$ as stated and assume that $f$ is nonnegative. Then 
\begin{equation*} 
	\int_{\rn} f^p M(w^s\chi_Q)^{\frac 1s}
	\le c_1 \int_{2Q} f^p w +  c_2\sum_{i=1}^\infty \int_{2^{i+1} Q\setminus 2^i Q} f^p\left( \frac{w^s(Q)}{|2^{i}Q|} \right)^\frac{1}{s}.
\end{equation*}%
The first term gives immediately the desired estimate. As for the second one, we notice that 
\begin{equation*}
\left(\frac {w^s(Q)}{|Q|}\right)^{1/s}\le C\frac {w(Q)}{|Q|},
\end{equation*}  
because $w\in RH_\sigma\subset RH_s$ and also that $w\in A_1$ implies 
\begin{equation*}
\frac{w(2^{i+1}Q)} {|2^{i+1}Q|}\le Cw(x)\quad\text{a.e. }x\in 2^{i+1}Q.
\end{equation*}
Inserting these in the above expression  we obtain 
\begin{align} 
\begin{split} \label{p4E}
	\sum_{i=1}^\infty & \int_{2^{i+1} Q\setminus 2^i Q} f^p w \ \frac{|2^{i+1}Q|}{w(2^{i+1}Q)} \frac{w(Q)}{2^{in/s} |Q|}
  \\ & \le  c_1  \norm{f}{\Lmwe}^p \sum_{i=1}^\infty {w(2^{i+1}Q)}^{\frac{rp}{n}} w(Q) 2^{i\frac{n}{s'}}
  \\ & \le  c_2  \norm{f}{\Lmwe}^p {w(Q)}^{1+\frac{rp}{n}}\sum_{i=1}^\infty {2}^{i\frac{rp}{\sigma'}}2^{i\frac{n}{s'}},
\end{split} 
\end{align}
using that $w\in RH_{\sigma}$ and \eqref{p12}. If $rp<-n\sigma'/\nu'$, we can choose $s>\nu$ sufficiently small such that the geometric series converges. To reach the endpoint $rp=-n\sigma'/\nu'$, use the openness of the reverse H\"older classes.

In the case of the spaces $\SLmwe$, the left-hand side of \eqref{p4E} is bounded by  
\begin{equation*}
C \norm{f}{\SLmwe}^p \sum_{i=1}^\infty (2^{in} |Q|)^{1+\frac {rp}n} 2^{\frac{in}{s'}} \frac{w(Q)}{w(2^{i+1}Q)}.
\end{equation*}
Using \eqref{p12} the geometric series converges for $rp+n(1/s'+1/\sigma)<0$. This gives \eqref{emb2Enew} in the stated range of $r$ (using again the openness of the reverse H\"older classes to get the endpoint).

For the particular embeddings \eqref{emb1E} take as $Q$ the unit cube centered at the origin and $\alpha=n/s$.
\end{proof}
\begin{Rem} 
Since we assume $w^\nu\in A_{1}$, we automatically have $w\in RH_{\nu}$. But each such $w$ will be in a reverse H\"older class  $w\in RH_{\sigma}$ for some $\sigma>\nu$ and the range of values of $r$ obtained in the proposition is not empty.
In the case $w(x)=|x|^\beta$ or $w(x)=(1+|x|)^\beta$, we need $-n<\beta\nu\le 0$ to fulfill the condition $w^\nu\in A_{1}$ and in such case the range of values of $r$ satisfying the two embeddings of \eqref{emb1E} are respectively
\[
		-\frac{n}{p} \le r<-\frac{n}{p} \frac{n}{n+\beta} \frac{1}{\nu'} \quad\text{and}\quad -\frac{n}{p} \le r<\frac{\beta}{p} -\frac{n}{p}\frac{1}{\nu'}. 	
\]
\end{Rem}

\begin{Corol}\label{newcor}
 For every $1\le \gamma\le p_0<\infty$, every $\gamma<p<\infty$ (and also $p=p_0$ if $\gamma=p_0$), $-n/p\le r<0$ and $w\in A_{p/\gamma}$, it holds that
 \begin{equation}\label{arb}
\Lmwe\subset \bigcup_{u\in A_{p_0/\gamma}}L_{p_0,u}(\rn).
\end{equation} 
If moreover $w\in A_{p/\gamma}\cap RH_{\sigma}$, for every $-n/p\le r< -n/(p \sigma) $ it holds that
 \begin{equation}\label{arb2}
\SLmwe\subset \bigcup_{u\in A_{p_0/\gamma}}L_{p_0,u}(\rn).
\end{equation} 
\end{Corol}

\begin{proof}
If $1\le \gamma< p_0$ we apply the embedding for $\Lmwe$ given by \eqref{emb1new} of Proposition \ref{LEmb} together with Proposition  \ref{DfOp}. If $\gamma=p_0$ and $p=p_0=\gamma$, \eqref{arb} follows from Proposition \ref{LEmbE} with $\nu=1$. Finally, if $\gamma=p_0$ and $p_0=\gamma<p<\infty$ we use the embedding for $\Lmwe$ of Proposition \ref{LEmb}  with $q=\gamma$.  Analogously, we obtain \eqref{arb2}.
\end{proof}

\begin{Rem}
 Note that as a consequence of this corollary, the left-hand side of \eqref{DefOp} in Proposition \ref{DfOp} also coincides with the (formally larger) union 
 \[\bigcup_{\gamma<p<\infty} \bigcup_{-n/p\le r<0} \bigcup_{w\in A_{p/\gamma}} \Lmwe.\]
(We recall that ${L}^{-{n}/{p}}_p(w,\rn)=L_{p, w}(\rn)$.)
\end{Rem}

\begin{Prop} \label{LEmbP2}
Let $1\le \gamma< b<\infty$ and $\sigma(p)=(b/p)'$ for $\gamma<p<b$. Let $w\in A_{p/\gamma}\cap RH_{\sigma}$ for $\sigma>\sigma(p)$. Let $r\in \left[-n/p , -n\sigma'/b\right]$. Then there exist $q_0\in(\gamma, p)$ and for each $q\in [\gamma,q_0]$, some $s_0(q)>\sigma (q)$, such that for $\sigma(q)<s\le s_0(q)$, each cube $Q$ and $h\in L_{(p/q)',w}(Q)$ with norm $1$, the continuous embedding
\begin{\eq} \label{emb11new}
  \Lmwe \hra L_{q,M(h^sw^s\chi_Q)^{1/s}}(\rn)
\end{\eq} 
holds with constant independent of $h$ and depending on $Q$ as $w(Q)^{1/p+r/n}$.

Furthermore, with the additional assumption $r\in \left[-n/p ,-n/(p\sigma) - n/b \right]$, it also holds that
\begin{\eq} \label{emb22new}
  \SLmwe \hra L_{q,M(h^sw^s\chi_Q)^{1/s}}(\rn)
\end{\eq} 
 with constant independent of $h$ and depending on $Q$ as $|Q|^{1/p+r/n}$.

As a consequence, under the same conditions on $p$, $r$ and $w$, there exists $0<\alpha<n/\sigma (q)$ such that 
\begin{\eq} \label{emb1P2}
\Lmwe,  \SLmwe \hra L_{q,(1+|x|)^{-\alpha}}(\rn). 
\end{\eq}
\end{Prop}

\begin{proof}
We can consider $\gamma=1$. Once this is proved, the case $\gamma>1$ is obtained from it by considering $|f|^\gamma$ as in Proposition \ref{LEmb} (see \eqref{TFE}).
 
Let $f\in \Lmwe$ nonnegative with  $w\in A_p\cap RH_{\sigma}$. 
Choose $q\in(1,p)$ such that $w\in A_{p/q}$ and set $\tilde p=p/q$. By \eqref{apandrh}, $w\in A_{\tilde p}\cap RH_{\sigma}$ implies $w^\sigma \in A_{\sigma(\tilde p-1)+1}$, which in turn implies $w^{1-\tilde p'}\in A_{1+(\tilde p'-1)/\sigma}$ by the duality property of the weights. If $1+(\tilde p'-1)/\sigma\le \tilde p'/s$, then $w^{1-\tilde p'}\in A_{\tilde p'/s}$. This is possible with $s>\sigma (q)$ if
\begin{equation*}
1+\frac{\tilde p'-1}{\sigma}<\frac{\tilde p'}{\sigma(q)}=\frac{\tilde p'}{(b/q)'},
\end{equation*}
which is the same as
\begin{equation*}
\frac 1{\tilde p'}+\frac1{\tilde p \sigma}<\frac{1}{(b/q)'}= 1-\frac qb,
\end{equation*} 
that is, $\sigma>\sigma(p)$.
 
The proof continues as for Proposition \ref{LEmb}. In particular, the estimate \eqref{p3} remains true because $w^{1-\tilde p'}\in A_{\tilde p'/s}$, so that we will get \eqref{emb11new} if \eqref{p5} holds. Since $w\in  RH_{\sigma}$, we use \eqref{p12} and the left-hand side of \eqref{p5} is bounded by
\begin{equation*} 
 c \, w\left( Q \right)^{\left(\frac{1}{p}+\frac{r}{n}\right)q} \sum_{i=1}^\infty 2^{i\left(\frac{rq}{\sigma'}+\frac n{s'}\right)},\end{equation*}
The geometric series converges for $r < -n\sigma'/(qs')$. It is possible to have $s>\sigma(q)$, if $r<-n\sigma'/b$. For the endpoint $r=-n\sigma'/b$, use the openness property of the reverse H\"older classes.

In the case of the spaces $\SLmwe$, to get \eqref{emb22new} we follow again the proof of Proposition \ref{LEmb} and we are left with \eqref{p10}. We insert \eqref{p12} and the geometric series converges if
\begin{equation*}
\frac{nq}p+r+\frac{n}{s'}-\frac{nq}{p\sigma'}=\frac{nq}{p\sigma}+r+\frac{n}{s'}<0.
\end{equation*}
We can get $s>\sigma(q)$ if 
\begin{equation*}
r<-\frac{nq}{p\sigma}-\frac{n}{\sigma(q)'}=-\frac{nq}{p\sigma}-\frac{nq}{b}.
\end{equation*}
This holds for $q$ close to $1$ if $r<-n/(p\sigma) - n/b$. 

Take as $Q$ the unit cube and as $h=cw^{-1}$ to get the embeddings of \eqref{emb1P2} with $\alpha=n/s$.  
\end{proof}
\begin{Prop} \label{LEmbP2B}
Let $1\le \gamma< b<\infty$ and $\sigma(p)=(b/p)'$ for $\gamma<p<b$. Let $\gamma<q<b$.
Then
\begin{\eq*} 
  \bigcup_{w\in A_{q/\gamma}\cap RH_{\sigma(q)}} L_{q,w}(\rn) = \bigcup_{\gamma< p<b} \bigcup_{w\in A_{p/\gamma}\cap RH_{\sigma(p)}} \Lpw. 
\end{\eq*}
\end{Prop}
\begin{proof}
We assume $\gamma=1$ as in the previous proposition. 
We notice first that using the factorization theorem of $A_p$ weights we have 
\begin{equation*}
A_p\cap RH_{\sigma(p)}=\{w: w^{\sigma(p)} \in A_{\sigma(p)(p-1)+1} \}
=\{u_0^{1-\frac pb}u_1^{1-p}: u_0, u_1\in A_1\}.
\end{equation*}

Let $f\in L_{p,w}(\rn)$ with  $w=u_0^{1-p/b}u_1^{1-p}$ for some $u_0, u_1\in A_1$, and let $1<p<q<b$. We have $fu_0^{-1/b}\in L^p (u_0 u_1^{1-p})$. Since $u_0 u_1^{1-p}\in A_p$, there exists $s>1$ such that $u_0 u_1^{1-p}\in A_{p/s}$. Then $(fu_0^{-1/b})^s\in L^{p/s} (u_0 u_1^{1-p})$. Hence $M((fu_0^{-1/b})^s)^{1/s}\in A_1$ and 
\begin{equation*}
|f(x)|\le M((fu_0^{-\frac 1b})^s)^{\frac 1s}(x) u_0^{\frac 1b}(x)\text{ a.e.}
\end{equation*}
We deduce that 
\begin{equation*}
\int_{\rn} |f|^q [M((fu_0^{-\frac 1b})^s)^{\frac 1s} u_0^{\frac 1b}]^{p-q} u_0^{1-\frac pb}u_1^{1-p} \le 
\int_{\rn} |f|^p u_0^{1-\frac pb}u_1^{1-p}.
\end{equation*}
The weight of the left side is  
\begin{equation*}
u_0^{1-\frac qb} [M((fu_0^{-\frac 1b})^s)^{\frac 1s}]^{p-q} u_1^{1-p}=u_0^{1-\frac qb} [v^\theta u_1^{1-\theta}]^{1-q},
\end{equation*}
where $v=M((fu_0^{-1/b})^s)^{1/s}$ and $\theta=(q-p)/(q-1) \in (0,1)$. Thus, it is in $A_q\cap RH_{\sigma(q)}$, since $v^\theta u_1^{1-\theta}\in A_1$. This proves that  
\begin{equation*}
\bigcup_{1< p\le q} \bigcup_{w\in A_p\cap RH_{\sigma(p)}} \Lpw =  \bigcup_{w\in A_q\cap RH_{\sigma(q)}} L_{q,w}(\rn).
\end{equation*}
This is enough to prove the statement because for $p>q$ we know from Proposition \ref{LEmbP2} (taking $r=-n/p$) that a function in $L_{p,w}(\rn)$ is in some $L_{\tilde q,(1+|x|)^{-\alpha}}(\rn)$ with $\tilde q$ near $1$ and the weight $(1+|x|)^{-\alpha}$ in $A_1 \cap RH_{\sigma(\tilde q)}\subset A_{\tilde q}\cap RH_{\sigma(\tilde q)}$.
\end{proof}

Our embeddings directly apply to similar results with respect to $A_p$-weighted weak $L_p(\rn)$ spaces ($W\Lpw$ is the space $W\Lmwe$ of Definition \ref{d1:def} for $r=-n/p$), which are nonseparable. These kind of embeddings allow one to extend operators to a domain which is a weak $L_p$ space via restriction. The connection with the Morrey spaces $\Lmwe$ lies in the use of a different norm for $W\Lpw$. 

For $1<p<\infty$ the space $W\Lpw$ can be normed with the norm
\begin{equation} \label{2:GG}
		\left\|f|W\Lpw\right\|_*\equiv  
\sup_{M} w(M)^{\left(\frac{1}{p}-\frac{1}{\tilde p} \right)} \left\|f|L_{\tilde p, w}(M)\right\|,
	\end{equation}	
where $0<\tilde p<p$ and the supremum is taken over all measurable subsets of $\real^n$ with $0<w(M)<\infty$. This norm is equivalent to the usual quasi-norm given in Definition \ref{d1:def} (\cite[Exercise 1.1.12]{G08}).	It is obvious from \eqref{2:GG} that
\begin{equation*}
W\Lpw \hra L_{\tilde p}^{r}(w,\rn),
\end{equation*}
with $r=-n/p>-n/\tilde p$.

Moreover, if $w\in A_{p/\gamma}$ for some $\gamma\in [1,p)$, $\tilde p$ can be chosen such that $w\in A_{\tilde p/\gamma}$, and Proposition \ref{LEmb} can be used to deduce that there exists $q_0>\gamma$ such that for each $q\in [\gamma,q_0]$ there is $0<\alpha<n$ for which
\begin{\eq*} 
  W\Lpw \hra L_{q,(1+|x|)^{-\alpha}}(\rn).
\end{\eq*} 
The same type of embeddings can be obtained from Propositions \ref{LEmbE} and \ref{LEmbP2}. 

\section{Boundedness on Morrey spaces: general results}\label{sec4}

In this section we obtain the general theorems giving the boundedness on the corresponding Morrey spaces from the assumption of weighted estimates on $\Lpw$ spaces. Instead of using operators, we can establish the general theorems in terms of pairs of functions, as it is usual in the recent presentations of the extrapolation theorems. This has the advantage of providing immediately several different versions.

\begin{Thm} \label{TE1}
Let $1\le p_0<\infty$ and let $\mathcal F$ be a collection of nonnegative measurable pairs of functions. Assume that for every $(f,g)\in \mathcal F$ and every $w\in A_{p_0}$ we have
\begin{\eq} \label{1exa1}
  \norm{g}{L_{p_0,w}(\rn)}\le c_1 \norm{f}{L_{p_0,w}(\rn)},
\end{\eq}
where $c_1$ does not depend on the pair $(f,g)$ and it depends on $w$ only in terms of $[w]_{A_{p_0}}$. 
Then for every
$1< p<\infty$ (and also for $p=1$ if $p_0=1$), every $-n/p \le r<0$ and every $w\in A_{p}$ we have
\begin{\eq} \label{2exa2}
  \norm{g}{\Lmwe}\le c_2 \norm{f}{\Lmwe}.
\end{\eq}

Furthermore, for every
$1< p<\infty$ (and also for $p=1$ if $p_0=1$) and $w\in A_{p} \cap RH_{\sigma}$, if $-n/p \le r\le-n/(p\sigma)$  we have
\begin{\eq} \label{3exa3}
  \norm{g}{\SLmwe}\le c_3 \norm{f}{\SLmwe}. 
\end{\eq}

The constants $c_2$ and $c_3$ in \eqref{2exa2} and \eqref{3exa3} 
do not depend on the pair $(f,g)$ but may depend on $w$ and the involved parameters.

If \eqref{1exa1} is assumed to hold only when the left-hand side is finite, \eqref{2exa2} and \eqref{3exa3} still hold assuming that their left-hand side is finite.
\end{Thm}

\begin{Rem}
In \eqref{1exa1} we do not need to require that $f$ and $g$ are in $L_{p_0,w}(\rn)$ for all $w\in A_{p_0}$. In the first part of the statement we assume that \eqref{1exa1} holds whenever the right-hand side is finite (and in such case, this forces the left-hand side to be finite). Then the conclusion of the theorem is that for every $f$ in $\Lmwe$ or in $\SLmwe$, $g$ is in the same space and the norm inequality is satisfied. When $g=Tf$ for some operator $T$ this means that $Tf$ is in the same Morrey space as $f$ is, for all the functions in the space. 

In the last part of the statement we only assume that \eqref{1exa1} holds when the left-hand side is finite. That is, it could happen that for a particular weight and a particular pair of functions the left-hand side is infinity and the right-hand side is finite. Under this less restrictive hypothesis, the fact that $f$ is in $\Lmwe$ or in $\SLmwe$ does not imply the same thing for $g$, but if $g$ also is assumed to be in the corresponding Morrey space, then the norm inequality holds.

In either case the extrapolation theorem in the usual weighted Lebesgue spaces theorem says that if \eqref{1exa1} holds for a particular $p_0$, it holds for $1<p<\infty$ and $A_p$ weights with the same interpretation of the inequalities. The proof given in \cite{D11}, for instance, is only valid in the first situation, although this is not explicitely indicated in the paper. The proofs in \cite{CMP11} cover the second case (left-hand side finite), and are valid also for the first case. In the applications to operators, the advantage of using the first version is that one proves directly the result for the whole weighted Lebesgue space, without making the extension from a dense subclass.
\end{Rem}

The part of the theorem corresponding to \eqref{2exa2} in the case $p>1$ was proved in \cite{RS16} using a different method, which involves the predual Morrey spaces.

\begin{proof} (a) \textit{The case $p>1$}. From the extrapolation theorem we can assume \eqref{1exa1} for any $p\in (1,\infty)$ and the weight class $A_p$.	

Let $w\in A_{p}$. Choose $q\in (1,p)$ and  $s>1$ such that the embedding \eqref{emb1new} holds. Set $\tilde p=p/q$.

Let $Q$ be a fixed cube.  Since
\begin{equation*}
\left(\int_{Q} g^{p}w\right)^\frac{1}{p}  = \left(\int_{Q} g^{\tilde p q}w\right)^\frac{1}{\tilde p q}
	=\sup_{h\,:\,\Vert h| L_{\tilde p',w}(Q)\Vert=1}\left(\int_{Q} g^{q}h w\right)^\frac{1}{q},
	\end{equation*}
we fix such a function $h$ and we have
\begin{align} \label{p2}
\begin{split}
\left(\int_{\rn} g^{q}h w \chi_{Q}\right)^\frac{1}{q} & \le  \left(\int_{\rn} g^{q}M(h^s w^s \chi_{Q})^\frac{1}{s}\right)^\frac{1}{q}\\
 & \le c \left(\int_{\rn} f^{q}M(h^s w^s \chi_{Q})^\frac{1}{s}\right)^\frac{1}{q},  
  \end{split}
\end{align}
where the second inequality holds because $M(h^s w^s \chi_{Q})^{1/s}\in A_1\subset A_q$. We checked in the proof of Proposition \ref{LEmb} that the weight is well defined. Note that the constant in \eqref{p2} is independent of $Q$ and $h$ because the $A_1$-constant of $M(h^s w^s \chi_{Q})^{1/s}$, which is an upper bound of the $A_q$-constant,  depends only on $s$ (actually, it behaves as $(s-1)^{-1}$; see \cite[Theorem 9.2.7]{G09}).

We are now in the situation of the embedding \eqref{emb1new} and the proof of \eqref{2exa2} for $p>1$ is complete.

In the case of the spaces $\SLmwe$, we use the embedding \eqref{emb2new}.

(b) \textit{The case $p=1$}. We assume now that \eqref{1exa1} holds for $p_0=1$ and $w\in A_1$ and prove the theorem for $p=1$ and $w\in A_1$.

We want to apply the embeddings of Proposition \ref{LEmbE} with $\nu=1$. Choose  $s>1$ such that \eqref{emb1Enew} and \eqref{emb2Enew} hold. Fix a cube $Q\subset\rn$. The weight $M(w^s\chi_Q)^{1/s}$ is in $A_1$. Then we have 
\begin{\eq*} 
	 \int_{Q} gw \le \int_{Q} g M(w^s\chi_Q)^{\frac 1s} \le c_1 \int_{\rn} f M(w^s\chi_Q)^{\frac 1s}.
\end{\eq*}%
Thus, \eqref{1exa1} and \eqref{2exa2} follow from \eqref{emb1Enew} and \eqref{emb2Enew}.

(c) If the assumption \eqref{1exa1} holds only when the left-hand side is finite, the proof is still valid assuming that the left-hand side of \eqref{2exa2} or  \eqref{3exa3} is finite. Indeed, since $g$ is in $\Lmwe$ or in $\SLmwe$, the embeddings \eqref{emb1new} and \eqref{emb2new} say that the second term in \eqref{p2} is finite and the proof of part (a) goes on. The same thing is true for part (b) of the proof using the embeddings of Proposition \ref{LEmbE}.
\end{proof}

There are weighted inequalities between pairs of operators with $L_{p,w}$-quasinorms for $0<p<\infty$ and $w\in A_\infty$. Such inequalities hold when the left-hand side is finite. We can get similar type of estimates with Morrey quasinorms.

\begin{Corol} \label{TEAinf}
Let $0< p_0<\infty$ and let $\mathcal F$ be a collection of nonnegative measurable pairs of functions. Assume that for every $(f,g)\in \mathcal F$ and every $w\in A_{\infty}$ we have
\begin{\eq} \label{1exa1inf}
  \norm{g}{L_{p_0,w}(\rn)}\le c_1 \norm{f}{L_{p_0,w}(\rn)},
\end{\eq}
whenever the left-hand side is finite. 
Then for every
$0< p<\infty$, every $-n/p \le r<0$ and every $w\in A_{\infty}$ we have
\begin{\eq} \label{2exa2inf}
  \norm{g}{\Lmwe}\le c_2 \norm{f}{\Lmwe},
\end{\eq}
whenever the left-hand side is finite. 

Furthermore, if $w\in RH_\sigma$ and $-n/p \le r\le-n/(p\sigma)$  we have
\begin{\eq} \label{3exa3inf}
  \norm{g}{\SLmwe}\le c_3 \norm{f}{\SLmwe}, 
\end{\eq}
  whenever the left-hand side is finite. 
\end{Corol}

\begin{proof}
By extrapolation in $L_{p,w}$ spaces we can assume that \eqref{1exa1inf} holds for all $p_0\in (0,\infty)$ (see \cite{CMP04}).

Let $0<p<\infty$. Given  $w\in A_{\infty}$, take $q\in [1,\infty)$ such that $w\in A_q$. If $p\ge q$, since $A_q\subset A_p$, the result follows directly from Theorem \ref{TE1}. If $p<q$, we have
\begin{equation*}
 \norm{g^{p/q}}{L_{1,u}(\rn)}\le c \norm{f^{p/q}}{L_{1,u}(\rn)},
\end{equation*}
for $u\in A_1$, whenever the left-hand side is finite, and we apply Theorem \ref{TE1} to obtain
\begin{equation*}
  \norm{g^{p/q}}{{ L}^{rp/q}_{q}(w,\rn)}\le c_2 \norm{f^{p/q}}{{ L}^{rp/q}_{q}(w,\rn)},
\end{equation*}
which is the desired result using the scaling \eqref{TFE} with $\gamma=p/q$. 

The proof of \eqref{3exa3inf} is similar.
\end{proof}

\begin{Corol} \label{TE2}
Let $1\le \gamma \le p_0<\infty$.
Assume that $T$ is an operator acting from $\bigcup_{w\in A_{p_0/\gamma}} L_{p_0,w}(\rn)$ into the space of measurable functions satisfying
\begin{\eq} \label{ex1}
  \norm{Tf}{L_{p_0,w}(\rn)}\le c_1 \norm{f}{L_{p_0,w}(\rn)} 
\end{\eq}
 for all $f\in L_{p_0,w}(\rn)$ and $w\in A_{p_0/\gamma}$, with a constant depending on $[w]_{A_{p_0}/\gamma}$.
Then for every
$\gamma< p<\infty$ (and also $p=\gamma$ if $p_0=\gamma$), every ${-n/p \le r<0}$ and every $w\in A_{p/\gamma}$,  we have
that $T$ is well-defined on $\Lmwe$ by restriction and, moreover,
\begin{\eq} \label{ex2}
  \norm{Tf}{\Lmwe}\le c_2 \norm{f}{\Lmwe}, 
\end{\eq}
for all $f\in \Lmwe$.

Furthermore, for $p$ as before and every $w\in A_{p/\gamma} \cap RH_{\sigma}$, if $-n/p  \le r\le-n/(p\sigma)$ we have
that $T$ is well-defined on $\SLmwe$ by restriction and, moreover,
\begin{\eq} \label{3exa3A}
  \norm{Tf}{\SLmwe}\le c_3 \norm{f}{\SLmwe} 
\end{\eq}
for all $f\in \SLmwe$.
The constants $c_2$ and $c_3$ in \eqref{ex2} and \eqref{3exa3A}
do not depend on $f$ but may depend on the weight $w$ and the involved parameters.

If $T$ satisfies instead of \eqref{ex1}  the weak-type assumption
\begin{equation*}
  \norm{Tf}{WL_{p_0,w}(\rn)}\le c_1 \norm{f}{L_{p_0,w}(\rn)} 
\end{equation*}
the estimates \eqref{ex2} and \eqref{3exa3A} are replaced by their weak-type counterparts, that is, $\norm{Tf}{W\Lmwe}$ and $\norm{Tf}{W\SLmwe}$ appear at the left-hand side of the inequalities. 
\end{Corol}

\begin{proof}
The operator is well defined by restriction as a consequence of the embeddings of Corollary \ref{newcor}. To obtain the estimates \eqref{ex2} and \eqref{3exa3A} it is enough to apply Theorem \ref{TE1} to the pairs $(|f|^\gamma, |Tf|^\gamma)$ with \break $f\in\bigcup_{w\in A_{p_0/\gamma}} L_{p_0,w}(\rn)$ and to use the scaling property of the norms in  
$\Lmwe$ and $\SLmwe$ (see \eqref{TFE}).

To deal with the weak-type inequalities we apply Theorem \ref{TE1} to the pairs $(|f|^\gamma, t^\gamma\chi_{\{|Tf|>t\}})$.
\end{proof}

\begin{Rem} 
	That the operator is  well-defined, for instance, on the set $\bigcup_{w\in A_{p}} \Lpw$ for some $1<p<\infty$ is for practical reasons not an additional condition. Usually one applies the theorem to continuous operators which are defined on an common dense subset of $\bigcap_{w\in A_{p}} \Lpw$, for example singular integrals defined on $\Dcal(\rn)$. The singular integrals can be extended to each $\Lpw$ by continuity. Such extensions coincide for  functions in the intersection of two weighted Lebesgue spaces, because for each such function one can find a  sequence in $\Dcal(\rn)$ converging to it in both spaces. This implies the well-definedness of the singular integrals on $\bigcup_{1<p<\infty, w\in A_{p}} \Lpw$. 
		\par
Concerning the spaces $\SLmwe$ it arises the question about the necessity of the additional weight condition, that is, whether for fixed $r$ being in some reverse H\"older class is necessary. It is known that our result could not hold for all $A_p$ weights (cf.~\cite{Sam09}). 
\end{Rem}
\begin{Rem}
Chiarenza and Frasca in \cite{ChFr87} noticed that weighted $A_1$ inequalities for singular integrals yields boundedness on unweighted Morrey spaces. With a different approach Adams and Xiao \cite{AX12} revisited this method using weighted inequalities to obtain boundedness in the unweighted Morrey spaces. 
\end{Rem}

\begin{Corol}\label{vvcorol}
Assume that we have a sequence of operators $\{T_j\}$ satisfying the assumptions of 
Corollary \ref{TE2}, and that the constant $c_1$ in \eqref{ex1} is independent of $j$. Then for $\gamma<q, p<\infty$ and the same conditions on $w$ and $r$, we have the vector-valued estimates
\begin{equation*} 
  \bigg\Vert{\big(\sum_j |T_jf_j|^q\big)^{1/q}}\bigg|{\Lmwe}\bigg\Vert\le c_4 \bigg\Vert{\big(\sum_j |f_j|^q\big)^{1/q}}\bigg|{\Lmwe}\bigg\Vert, 
\end{equation*} 
and
\begin{equation*} 
   \bigg\Vert{\big(\sum_j |T_jf_j|^q\big)^{1/q}}\bigg|{\SLmwe}\bigg\Vert
   \le c_5 \bigg\Vert{\big(\sum_j |f_j|^q\big)^{1/q}}\bigg|{\SLmwe}\bigg\Vert. 
\end{equation*}   
 \end{Corol}
To prove these vector-valued inequalities in Morrey spaces it is enough to apply Theorem \ref{TE1} to pairs $\left(\big(\sum_j |f_j|^q\big)^{1/q}, \big(\sum_j |T_jf_j|^q\big)^{1/q}\right)$.

Some operators satisfy weighted inequalities for classes of weights of the type $A_q\cap RH_\sigma$. This is the case of adjoints of linear operators bounded on $\Lpw$ for $p>\gamma$ and $w\in A_{p/\gamma}$. By duality the adjoint operator is bounded on $\Lpw$ for $1<p\le \gamma$ and $\{w^{1-p}: w\in A_{p'/\gamma}\}$, which is the same as  $A_{p}\cap RH_{{\gamma'}/{(\gamma'-p)}}$. Also operators bounded in a limited range of values of $p$ satisfy weighted inequalities for weights in classes of such type (see \cite{AM07} and \cite{CDL12}). 
Similarly, we have sometimes the weak or strong boundedness with weights in a class of the type $\{w: w^{\nu}\in A_1\}$. 

\begin{Thm} \label{TE1P}
Let $1\le p_-\le p_0\le p_+<\infty$ and let $\mathcal F$ be a collection of nonnegative measurable pairs of functions. Let $\mathcal F$ be a collection of nonnegative pairs of function in $L_{p_0,w}(\rn)$. Assume that for every $(f,g)\in \mathcal F$ and every $w\in A_{p_0/p_-}\cap RH_{(p_+/p_0)'}$ we have
\begin{\eq} \label{1exa1PP}
  \norm{g}{L_{p_0,w}(\rn)}\le c_1 \norm{f}{L_{p_0,w}(\rn)},
\end{\eq}
where $c_1$ depends on the $A_{p_0}$-constant of the weight $w$. 
Let  $p_-< p<p_+$ (and $p=p_-$ if $p_0=p_-$) and $w\in A_{p/p_-}\cap RH_\sigma$ with $\sigma\ge (p_+/p)'$. Then for every $-n/p \le r\le -n\sigma'/p_+$ we have
\begin{\eq} \label{2exa2P}
  \norm{g}{\Lmwe}\le c_2 \norm{f}{\Lmwe}. 
\end{\eq}

Furthermore,  if $-n/p \le r\le-n/(p\sigma)-n/p_+$  we have
\begin{\eq} \label{3exa3P}
  \norm{g}{\SLmwe}\le c_3 \norm{f}{\SLmwe}.
\end{\eq}

The constants $c_2$ and $c_3$ in  \eqref{2exa2P} and \eqref{3exa3P} 
do not depend on $f$ but may depend on the weights and the involved parameters.
\end{Thm}

\begin{proof}
It is enough to prove the result for $p_-=1$, which we assume in what follows. To treat the case $p_->1$ write \eqref{1exa1PP} for the pair $(f^{p_-},g^{p_-})$ on  $L_{p_0/p_-,w}(\rn)$, and apply the case $p_-=1$.

(a) \textit{The case $p>p_-$} (that is, $p>1$). First we observe that by limited range extrapolation (cf.~\cite[Thm.~4.9]{AM07} or \cite[Thm.~3.31]{CMP11}) \eqref{1exa1PP} holds in $L_{p,v}(\rn)$ with $v\in A_{p}\cap RH_{(p_+/p)'}$ for $1<p<p_+$.   

Let $w\in A_{p}\cap RH_{\sigma}$. We will apply Proposition \ref{LEmbP2} with $b=p_+$ and $\gamma=1$. We choose $q\in (1,p)$ and  $s>(p_+/q)'$ such that the embedding \eqref{emb11new} holds.  Let $Q$ be a fixed cube. As in the proof of Theorem \ref{TE1}, we use duality, and for $h$ in $L_{\tilde p',w}(Q)$ of norm $1$ we get \eqref{p2}, that is,
\begin{equation*} 
	\left(\int_{\rn} g^{q}h w \chi_{Q}\right)^\frac{1}{q}
  \le c \left(\int_{\rn} f^{q}M(h^s w^s \chi_{Q})^\frac{1}{s}\right)^\frac{1}{q}, 
\end{equation*}
with constant independent of $h$ and $Q$, because $M(h^s w^s \chi_{Q})^{1/s}\in A_{1}\cap RH_{(p_+/q)'}$, due to the choice of $s>(p_+/q)'$. Under the conditions required to $r$ we can apply Proposition \ref{LEmbP2} to obtain \eqref{2exa2P} and  \eqref{3exa3P}.

(b) \textit{The case $p=p_-$} (that is, $p=1$).
Now we are assuming that the inequality holds for $p=1$ and $w\in A_{1}\cap RH_{p_+'}$ (that is, $w^{p_+'}\in A_1$). The proof is as in part (b) of the proof of Theorem \ref{TE1} except for the fact that we need to choose $s>p_+'$ to guarantee that  $M(w^s \chi_{Q})^{p_+'/s}$ is in $A_1$. 

Under the assumptions on $r$, we are in the conditions of Proposition \ref{LEmbE} with $\nu=p_+'$. Then \eqref{2exa2P} and  \eqref{3exa3P} are deduced from \eqref{emb1Enew} and \eqref{emb2Enew} with $p=1$.
\end{proof}
Using this theorem and the embeddings given in Propositions \ref{LEmbE} and \ref{LEmbP2} we can obtain results analogous to those of Corollary \ref{TE2}, for the corresponding ranges of $p$ and classes of weights.

\begin{Corol} \label{TE2P}
Let $1\le p_-\le p_0\le p_+<\infty$. 
Assume that $T$ is an operator acting from $\bigcup_{w\in A_{p_0/p_-}\cap RH_{(p_+/p_0)'}} L_{p_0,w}(\rn)$ into the space of measurable functions that satisfies
\begin{equation*} 
  \norm{Tf}{L_{p_0,w}(\rn)}\le c_1 \norm{f}{L_{p_0,w}(\rn)} 
\end{equation*}
 for all $f\in L_{p_0,w}(\rn)$ and $w\in A_{p_0/p_-}\cap RH_{(p_+/p_0)'}$, with a bound depending on the constant of the weight. 
Then for every $p_-< p<p_+$ (and also $p=p_-$ if $p_0=p_-$) and $w\in A_{p/p_-}\cap RH_\sigma$ with $\sigma\ge (p_+/p)'$, and for every $-n/p \le r\le -n\sigma'/p_+$ we have
\begin{equation*}
  \norm{Tf}{\Lmwe}\le c_2 \norm{f}{\Lmwe} 
\end{equation*}
for all $f\in \Lmwe$.

Furthermore,  if $-n/p \le r\le-n/(p\sigma)-n/p_+$  we have
\begin{equation*} 
  \norm{Tf}{\SLmwe}\le c_3 \norm{f}{\SLmwe} 
\end{equation*}
for all $f\in \SLmwe$.
\end{Corol}

Weak-type inequalities like in Corollary \ref{TE2} and vector-valued inequalities like in Corollary \ref{vvcorol}   can be written also in this setting.

\begin{Rem}
In \cite{CMP11} the authors get very general results with respect to extrapolation in Banach function spaces. But in this abstract setting they do not get results either at the endpoint ($A_1$ weights) or with respect to limited range extrapolation (cf.~Theorems \ref{TE1P}). They use a kind of duality (associated spaces) which is available for Morrey spaces of type $\Lmwe$ (cf.~\cite{RS16}), but we completely avoid it. However, their results deliver weighted inequalities which means that for applications to operators one still has to deal with the well-definedness of the operators (which we have done via embeddings).
\end{Rem}
\section{Applications to mapping properties of operators}\label{sec5}

There is a plentiful list of operators fulfilling the requirements of one or more of the theorems of the previous section. As we mentioned in the introduction several of those operators have been proved to be bounded in Morrey spaces (unweighted or weighted) using in general the particular form of the operator or some particular pointwise bound. In most cases only the size estimate with the Morrey norm has been proved, without any discussion about the possible definition of the operator in the corresponding Morrey space. All of this is avoided with our approach. In what follows, we present a collection of applications of our general results. Vector-valued Morrey inequalities also hold in all cases, but we do not mention them. In the last subsection we give Morrey estimates with $A_\infty$ weights and $0<p<\infty$.

\subsection{Calder\'{o}n-Zygmund operators and their maximal truncations}

It is by now a classical result that Calder\'on-Zygmund operators are bounded on $L_{p,w}(\rn)$ for $w\in A_p$ ($1<p<\infty$), and that they are of weak-type $(1,1)$ with respect to $A_1$ weights. Here we can understand the term Calder\'on-Zygmund operator in the general sense, that is, when the operator is represented by a two-variable kernel $K(x,y)$ with appropriate size and regularity estimates. The maximal operator associated to the  Calder\'on-Zygmund operators taking the supremum of the truncated integrals also satisfies similar weighted estimates. For a proof the reader can consult \cite[Chapter 9]{G09}. Then we can apply Corollary \ref{TE2}.

The regularity assumptions on $K(x,y)$ can be weakened and stated in some integral form called $L^r$-H\"ormander condition. In the convolution case it appears implicitely in \cite{KuW79}, and in a more general setting in \cite{RRT86}, for instance. In that case the weighted inequalities hold for $L_{p,w}(\rn)$ for $p\ge\gamma$ and $w\in A_{p/\gamma}$, where $\gamma=r'$. Then Corollary \ref{TE2} applies. If the estimates are symmetric in the variables $x$ and $y$ of the kernel, duality can be used for the weighted estimates and the results of Corollary  \ref{TE2P}  apply.

\subsection{Multipliers} 
Kurtz and Wheeden studied in \cite{KuW79} weighted inequalities for classes of multipliers defined as follows (see also \cite{StTo89}). Let $\widehat{Tf}(\xi)=m(\xi)\hat f(\xi)$. We say that $m\in M(s,l)$ if 
\begin{equation*}
\sup_{R>0} \left(R^{s|\alpha|-n}\int_{R<|\xi|<2R}|D^\alpha m(\xi)|^s\right)^{1/s}<+\infty \text{ for all }|\alpha|\le l.
\end{equation*}
The result of Kurtz and Wheeden says that for $1<s\le 2$ and $n/s<l\le n$, $T_m$ is bounded on $L_{p,w}(\rn)$ if $n/l<p<\infty$ and $w\in A_{pl/n}$; its dual result for $1<p<(n/l)'$; and of weak-type $(1,1)$ with respect to weights $w$ such that $w^{n/l}\in A_1$. Although they do not make it explicit, interpolation gives that $T_m$ is bounded on $L_{p,w}(\rn)$ for $w^{lp/n}\in A_1$, when $1<p\le n/l$. 
We can apply Corollaries \ref{TE2} and  \ref{TE2P}, depending on the values of $p$. 

A particularly interesting case is $M(2,l)$, which corresponds to the Mihlin-H\"ormander multipliers. For $l=n$ the full range of results holds applying Corollary \ref{TE2} with $\gamma=1$. Let us particularize some of the results valid for any $l>n/2$:
\begin{itemize}
\item for $p\ge 2$ and $w\in A_{p/2}$, we have boundedness on $ \Lmwe$ without restriction on $r$;  for $w\in A_{p/2}\cap RH_\sigma$ we have boundedness on $\SLmwe$ for $-n/p\le r\le -n/(p\sigma)$ (Corollary \ref{TE2P} with $p_-=2$);
\item for $1<p<2$ and $w\in A_{p}\cap RH_\sigma$ with $\sigma'<2/p$, we have boundedness on $\Lmwe$ when $r\in [-n/p,-n\sigma'/2]$ and on $\SLmwe$ for $--n/p\le r\le -n/(p\sigma)-n/2$ (Corollary \ref{TE2P});
\item for $p\ge 2$ and $w(x)=|x|^\beta$ (or $(1+|x|)^\beta$) with $-n<\beta<n(p/2-1)$, we have boundedness on $ \Lmwe$ without restriction on $r$ (Corollary \ref{TE2} with $\gamma=2$); 
\item for $p\ge 2$ and $w(x)=|x|^\beta$ (or $(1+|x|)^\beta$) with $-n<\beta<0$  we have boundedness on $ \SLmwe$ for  $-n/p\le r\le \beta/{p}$ and if $0\le\beta< n(p/2-1)$, we have boundedness on $\SLmwe$ for $-n/p \le r<0$ (Corollary \ref{TE2P} with $p_-=2$ together with Remark \ref{Ma3}); 
\item for $1<p\le 2$ and $w(x)=|x|^\beta$ (or $(1+|x|)^\beta$) with $-np/2<\beta\le 0$, the boundedness on  $\Lmwe$ holds for $-\frac np\le r< -\frac {n^2}{n+\beta}\left(\frac 1p-\frac 12\right)$. This is valid for $p=1$ using the weak Morrey estimate.   In the case of  $\SLmwe$ the range is  $-\frac np\le r<-n\left(\frac 1p-\frac 12\right)+\frac \beta p$. We apply Corollary \ref{TE2P} with $p_-=p$ and $(p_+/p)'=2/p$, and Remark \ref{Ma3}.
\item for $1<p\le 2$ and $w(x)=|x|^\beta$ (or $(1+|x|)^\beta$) with $0\le \beta<n(p-1)$, we have boundedness on $\Lmwe$ and on $\SLmwe$ when $r\in [-n/p,-n/2)$. Here we use Corollary \ref{TE2P} with $p_+=2$ and the fact that these weights are in $RH_\sigma$ for all $\sigma>1$. 
\end{itemize}
The origin of Morrey spaces is in the study of the smoothness properties of solutions of PDEs. For this reason results on the multipliers $M(s,l)$ related to their smoothness are interesting (cf. \cite{Tr12, T-HS, R}).  

Another interesting case of multipliers is that of Marcinkiewicz multipliers in one dimension. They are bounded on $L_{p,w}(\real)$ for $w\in A_p$ (see \cite{Ku80}) and Corollary \ref{TE2} gives all the Morrey boundedness results for $1<p<\infty$.

\subsection{Rough operators} 

The rough singular integrals are defined in general by
\begin{equation*}
T_{\Omega, h}f(x)=\text{p.v.}\int_{\rn} \frac{\Omega(y')h(|y|)}{|y|^n}f(x-y)dy
\end{equation*}
with $\Omega\in L^1(S^{n-1})$ ($y'=y|y|^{-1}$) and integral zero, and $h$ defined on $(0,\infty)$. The (weighted) boundedness of the operator is proved using additional assumptions on $\Omega$ and $h$. For instance, if both $\Omega$ and $h$ are in $L^\infty$, it was proved in \cite{DuRu86} that $T_{\Omega, h}$ is bounded on $L_{p,w}(\rn)$ for $w\in A_p$ ($1<p<\infty$), so that the full range is obtained for Morrey spaces from Theorem \ref{TE1}. 

If for fixed $q>1$ we know that $\Omega\in L^q(S^{n-1})$ and $h\in L^q((0,\infty);dt/t)$ then the boundedness on $L_{p,w}(\rn)$ holds for $w\in A_{p/q'}$ ($q'\le p<\infty$) and also for some classes of weights obtained by duality for the range $1<p\le q$. These classes can further be extended by interpolation. Results of this type are in \cite{Wat90} and \cite{D93}. The weighted inequalities are very much like those for the multipliers of the previous section, hence similar Morrey estimates are obtained. Due to the openness properties of $A_p$ weights, D.~Watson noticed that for $\Omega\in \bigcap_{q<\infty} L^q(S^{n-1})$ and $h\in \bigcap_{q<\infty}L^q((0,\infty);dt/t)$ the full range $L_{p,w}(\rn)$ for $w\in A_p$ and $1<p<\infty$ is also obtained. Hence the full range of Morrey estimates is deduced. Notice that in this case the kernel of the operator $T_{\Omega, h}$ need not satisfy the size estimate \eqref{i5}.

Similar weighted estimates are known for the maximal operator defined from the truncated integrals of the rough operators.

\subsection{Square functions} 

There exist several types of square functions for which our theorems apply. 

The classical Littlewood-Paley operators in one dimension associated to the dyadic intervals are bounded on $L_{p,w}(\real)$ for $w\in A_p$ (see \cite{Ku80}). A similar result holds in all dimensions in we consider the square functions of discrete or continuous type built using dilations of a Schwartz function with zero integral (see \cite[Prop.~1.9]{Ry01} and the comments following it).

The Lusin area integral, the $g_\lambda$ functions, Stein's $n$-dimensional extension of the Marcinkiewicz integral and the intrinsic square function of M. Wilson are other examples of square functions for which the weighted inequalities on $L_{p,w}(\rn)$ for $w\in A_p$ ($1<p<\infty$) are known to hold (see \cite{ToW90} and \cite[Thm.~7.2 and notes of Chapter 7]{Wil08}). Of course, all of them are covered by our Corollary \ref{TE2}.

Rough variants have been considered also for square functions with the introduction of some $\Omega\in L^q(S^{n-1})$ as for the singular integrals mentioned above. The class of weights for which boundedness holds depends on  $q$ (see \cite{DFP99} and \cite{DuS02}).

On the other hand, we can mention Rubio de Francia's results of Littlewood-Paley type. The square function built using comparable cubes in $\rn$ is bounded on $L_{p,w}(\rn)$ for $w\in A_{p/2}$ and $2\le p<\infty$ (see \cite{Ru83}) and in this case we can apply Corollary \ref{TE2} with $\gamma=2$ to obtain the corresponding Morrey estimates. For the square function built using arbitrary disjoint intervals in $\real$ we can apply Corollary \ref{TE2} to the result of \cite{Ru85} in which the boundedness of the operator is obtained for $L_{p,w}(\rn)$ with $w\in A_{p/2}$ and $p>2$ (the endpoint $p=2$ remains open).

\subsection{Commutators} 

There are several results on weighted inequalities for commutators to which we can apply our theorems to obtain the corresponding Morrey estimates.  Let us mention in particular the remarkable result of \'Alvarez et al.~in \cite{ABKP93}, which allows to obtain many weighted inequalities for commutators with BMO functions. A good list of applications can be seen in their paper. For each one of them we obtain the corresponding result in the Morrey spaces.

\subsection{Other operators} 

There are many more results in the literature concerning weighted inequalities with $A_p$ type weights. We briefly mention here some other remarkable operators for which our theorems apply directly. Although in some cases particular proofs of their Morrey estimates do exist, not all of them seem to have been considered before.

Oscillatory singular integrals with kernels of the form $e^{iP(x,y)}K(x,y)$ where $P(x,y)$ is a polynomial and $K(x,y)$ is a Calder\'on-Zygmund standard kernel or a kernel of rough type have been studied and weighted results appear for instance in \cite{LDY07}. A further variant are the so-called multilinear oscillatory singular integrals in which the kernel is modified by inserting a factor $|x-y|^{-m}R_{m+1}(A;x,y)$ where the second term is the $(m+1)$-th order remainder of the Taylor series of a function $A$ expanded in $x$ about $y$. Some weighted inequalities are in \cite{Wu04}.  Other types of oscillating kernels are of the form $K_{a,b+iy}(t)=\exp(i|t|^a)(1+|t|)^{-(b+iy)}$ for which weighted inequalities are obtained in \cite{ChKS83}, and strongly singular convolution operators corresponding to multipliers of the form $\theta(\xi)e^{i|\xi|^b}|\xi|^{-a}$, where $\theta$ is a smooth radial cut-off function (vanishing in a neighborhood of the origin and which is identically $1$ outside a compact set), studied in \cite{Ch84}. 

Bochner-Riesz operators at the critical index satisfy $A_p$-weighted estimates for $1<p<\infty$ and weak-type $(1,1)$ estimates with $A_1$ weights (\cite{ShSu92} and \cite{Var96}). The boundedness on the Morrey spaces are obtained from Corollary \ref{TE2}. Below the critical index only partial results can be obtained, but still there are some subclasses of $A_p$ weights (see \cite{CDL12}, for instance) for which the corresponding theorems in the previous section can be applied. Some weighted inequalities also hold for the maximal Bochner-Riesz operator. 

Weighted inequalities for pseudodifferential operators are for instance in \cite{Ya85}  and \cite{Sat05}, while for Fourier integral operators they can be found in \cite{DSS14}. In both cases, Morrey space estimates can be deduced from our results in Section \ref{sec4}. 

\subsection{Morrey norm inequalities with $A_\infty$ weights} 
Let $0<p<\infty$ and $w\in A_\infty$. Then the following inequalities hold whenever the left-hand side is finite:
\begin{align}
\label{five1}
\norm{Tf}{\Lmwe}&\le c_1 \norm{Mf}{\Lmwe},\\
\label{five2}
\norm{Mf}{\Lmwe}&\le c_2 \norm{M^\sharp f}{\Lmwe},\\
\label{five3}
\norm{I_\alpha f}{\Lmwe}&\le c_3 \norm{M_\alpha f}{\Lmwe},
\end{align}
where $T$ is a Calder\'on-Zygmund operator in \eqref{five1}, $M$ is the Hardy-Littlewood maximal operator in \eqref{five1} and  \eqref{five2}, $M^\sharp$ is the sharp maximal function (see \cite[p. 117]{D01}) in  \eqref{five2}, and $I_\alpha$ and $M_\alpha$ are the fractional integral and the fractional maximal operator of order $\alpha\in (0,n)$ (see \cite[p. 88-89]{D01}) in \eqref{five3}. Similar inequalities hold with $\SLmwe$ instead of $\Lmwe$ if $w$ is assumed to be in $RH_\sigma$ and $-n/p \le r\le-n/(p\sigma)$.

The inequalities are derived from their $L_{p,w}$ counterparts using Corollary \ref{TEAinf}. A proof of the needed $L_{p,w}$ estimates using extrapolation can be found in \cite{CMP04}, where references concerning the original inequalities (obtained by means of good-$\lambda$ inequalities) are given. In a similar way, other $L_{p,w}$ inequalities appearing in \cite{CMP04} can also be written with Morrey norms. In that paper weak-type inequalities and vector-valued inequalities are also given and they can also be transferred to the Morrey setting, but we do not go into this in detail. Weighted inequalities with Morrey norms involving $M^\sharp$ with $A_\infty$ weights may be found also in  \cite[Corollary 2]{NkSa17}.

\end{document}